%% file: Hyperuniformity.tex
\documentclass[12pt]{amsart}

\usepackage{amsmath}
\usepackage{amssymb,bbm,url}
\usepackage{mathrsfs}
\usepackage[usenames,dvipsnames]{color}
\usepackage{hyperref}
\usepackage{cite}

\newtheorem{lemma}{Lemma}
\newtheorem{theorem}{Theorem}

\theoremstyle{definition}
\newtheorem{definition}{Definition}
\newtheorem{remark}{Remark}
\newtheorem*{ackno}{Acknowledgement}

\DeclareMathOperator{\sgn}{sgn}
\DeclareMathOperator{\id}{id}

\DeclareMathOperator{\SO}{SO}
\newcommand{\Sph}{\mathbb{S}}
\newcommand{\Ind}{\mathbbm{1}}

\newcommand{\X}{\mathscr{X}}
\newcommand{\J}{\mathcal{J}}

\DeclareMathOperator{\diam}{diam}

\newcommand{\dd}{\mathrm{d}}

\begin{document}

\title[Hyperuniformity]
{Hyperuniform point sets on the sphere:\\ probabilistic aspects}





\author[J. S. Brauchart]{Johann S. Brauchart\textsuperscript{\dag}}
\author[P. J. Grabner]{Peter J. Grabner\textsuperscript{\ddag}}
\address[J. B., P. G.]{Institute of Analysis and Number Theory,
  Graz University of Technology,
  Kopernikusgasse 24.
8010 Graz,
Austria}
\email{j.brauchart@tugraz.at}
\thanks{\dag{} This author is supported by the Lise Meitner scholarship M~2030
of the Austrian Science Foundation FWF}

\thanks{\textsuperscript{\ddag} These authors were supported by the Austrian
  Science Fund FWF project F5503 (part of the Special Research Program (SFB)
  ``Quasi-Monte Carlo Methods: Theory and Applications'')}
\email{peter.grabner@tugraz.at}

\author[W. Kusner]{W\"oden Kusner\textsuperscript{\ddag}}
\address[W. K.]{Department of Mathematics, Vanderbilt University,
1326 Stevenson Center, Nashville, TN 37240, USA}
\email{wkusner@gmail.com}

\author[J. Ziefle]{Jonas Ziefle}
\address[J. Z.]{Fachbereich Mathematik,
  Auf der Morgenstelle 10, 72076 T\"ubingen, Germany}
\email{jonas.ziefle@googlemail.com}

\begin{abstract}
  The concept of hyperuniformity has been introduced by Torquato and Stillinger
  in 2003 as a notion to detect structural behaviour intermediate
  between crystalline order and amorphous disorder. The present paper studies a
  generalisation of this concept to the unit sphere. It is shown that several
  well studied determinantal point processes are hyperuniform.
\end{abstract}
\maketitle

\input{Introduction}
\input{PointProcesses}
\input{Definition}
\input{Intersection}
\input{Spherical}
\input{Harmonic}
\input{Jittered}

\begin{ackno}
  This material is based upon work supported by the National Science
  Foundation under Grant No. DMS-1439786 while the first three authors
  were in residence at the Institute for Computational and
  Experimental Research in Mathematics in Providence, RI, during the
  Spring 2018 semester.
\end{ackno}
\providecommand{\bysame}{\leavevmode\hbox to3em{\hrulefill}\thinspace}
\providecommand{\MR}{\relax\ifhmode\unskip\space\fi MR }
\providecommand{\MRhref}[2]{%
  \href{http://www.ams.org/mathscinet-getitem?mr=#1}{#2}
}
\providecommand{\href}[2]{#2}

\end{document}

%% file: Introduction.tex
\section{Introduction}
It has been observed for a long time in the physics literature that
large (ideally infinite) particle systems can exhibit structural
behaviour between crystalline order and total disorder. Very prominent
examples are given by quasi-crystals and jammed sphere packings. Research in
mathematics and physics has been inspired by the
discovery of such materials which lie between crystalline
order and amorphous disorder. We just mention de~Bruijn's Fourier analytic
explanation for the diffraction pattern of quasi-crystals \cite{dBr86}
and the extensive collection of articles on quasi-crystals
\cite{AxGr95} as examples.

Hyperuniformity was introduced in \cite{torquato2003local} as a concept to
measure the occurrence of ``intermediate'' order. Such hyperuniform
configurations $X$ occur in jammed packings, in colloids, as well as in
quasi-crystals. The main feature of hyperuniformity is the fact that local
density fluctuations (``number variance'') are of smaller order than for an
i.i.d. random (``Poissonian'') point configuration.


The point of view taken in \cite{torquato2003local} was probabilistic based on
point processes. It has since been observed that determinantal point processes
exhibit less disordered behaviour in comparison to i.i.d. points due to the
built in mutual repulsion of particles (see
\cite{Hough_Krishnapur_Peres+2009:zeros_gaussian}). The prototypical example of
such a point process is given by the distribution of fermionic particles, whose
joint wave function is given as a determinant expressed in terms of the
individual wave functions.

An infinite discrete point set $X\subset\mathbb{R}^d$ is defined to be
hyperuniform if the variance of the random variable (``number variance'')
$\#((\mathbf{x}+t\Omega)\cap X)$ behaves like $o(t^d)$ as $t\to\infty$. Here, 
$\Omega$ is a fixed compact test set (``window''); in most of the cases
$\Omega$ is chosen as a Euclidean ball. Notice that the number variance for
i.i.d. point sets is of exact order $t^d$. Thus, hyperuniformity is
characterised by a smaller order of magnitude of the variance. It was shown in
\cite{torquato2003local} that the best possible order for the variance is
$t^{d-1}$.

In \cite{Brauchart_Grabner_Kusner2019:hyperuniform_deterministic} a notion of
hyperuniformity for sequences of finite point sets on the sphere was
introduced. In that paper three regimes of hyperuniformity were identified and
studied, and several sequences of deterministically given point sets such as
designs, QMC-designs, and certain energy minimising point sets were shown to
exhibit hyperuniform behaviour. We also refer the reader to related recent work
\cite{PhysRevE.100.022107,PhysRevE.99.032601}.

It is the aim of the present paper to study hyperuniformity on the sphere for
samples of point processes on the sphere. Especially, we study the spherical
ensemble (see \cite{Krishnapur2006:zeros_random_analytic,
  Hough_Krishnapur_Peres+2009:zeros_gaussian}) on $\mathbb{S}^2$
(Section~\ref{sec:spherical-ensemble}), the harmonic ensemble introduced in
\cite{Beltran_Marzo_Ortega-Cerda2016:determinantal}
(Section~\ref{sec:harmonic}), and the jittered sampling process
(Section~\ref{sec:jittered}). We observe that the jittered sampling process can
be seen as a determinantal point process. All processes turn out to be
hyperuniform in all three regimes.  The harmonic ensemble has slightly weaker
behaviour in the threshold order regime.

Throughout this paper $\sigma=\sigma_d$ will denote the normalised
surface area measure on $\mathbb{S}^d$. We suppress the dependence on $d$ in
this notation.


%% file: PointProcesses.tex
\section{Point Processes}\label{sec:point-processes}

We consider a point process $\X_N$ sampling $N$ points given by the \emph{joint
  densities} $\rho^{(N)}$
in the sense that
\begin{equation*}
  \mathbb{P}\left((X_1,\dots,X_N)\in B\right)=
  \idotsint_B\rho^{(N)}(\mathbf{x}_1,\dots,\mathbf{x}_N)\,
  \dd\sigma (\mathbf{x}_1)
  \cdots \dd\sigma (\mathbf{x}_N),
\end{equation*}
where $B$ is a measurable subset of $(\Sph^d)^N$.
We will assume
throughout this paper that the number of points $N$ is fixed and that the
process is simple, which means that the probability of sampling a point more
than once is zero. In some of the studied examples the number of points will
depend on a parameter $L$; in these cases we write $N_L$ for this number.

Note that in the literature
(e.g.,~\cite{Hough_Krishnapur_Peres+2009:zeros_gaussian}) the process is often
given in terms of its \emph{joint intensities} (\emph{correlation functions})
which are given by $N!\cdot\rho^{(N)}$. We use joint densities with respect to
the natural measure $\sigma$ on $\Sph^d$ in this paper
since they make the asymptotic dependence on $N$ more transparent. By a result
of Lenard~\cite{Soshnikov2000}, locally integrable functions $\rho^{(N)}$ can be
represented as the joint densities of a point process if and only if they
satisfy a positivity condition and the particles are exchangeable; i.e., the
joint densities are invariant under permutation of the entries
\begin{align} \label{permutation-inv} 
  \rho^{(N)}(\mathbf{x}_{\tau(1)},\dots,\mathbf{x}_{\tau(N)})=
  \rho^{(N)}(\mathbf{x}_1,\dots,\mathbf{x}_N) \ 
\text{for all} \
\mathbf{x}_i\in \Sph^d, \ \tau \in \text{S}_N .
\end{align}  
The \emph{reduced densities}
\begin{align*} 
\rho^{(N)}_k(\mathbf{x}_1,\dots,\mathbf{x}_k):=\int_{(\Sph^d)^{N-k}}
  \rho^{(N)}(\mathbf{x}_1,\dots,\mathbf{x}_N)\, \dd\sigma (\mathbf{x}_{k+1})
  \cdots \dd\sigma (\mathbf{x}_N),
\end{align*}
$1 \leq k \leq N$, 
describe how $k$ of $N$ points are distributed. The joint intensities are $\frac{N!}{(N-k)!}\rho^{(N)}_k$. 

The number of points that are put into a test set
$B\subseteq\Sph^d$ by the process is the random variable
$\X_N(B):=\sum_{i=1}^N \Ind_B(X_i)$, or in other words $N$ times the
empirical measure of $B$. As usual, $\Ind_B$ denotes the indicator function of
the set~$B$.

For most of our study, we restrict ourselves to processes that are invariant
under isometries of the sphere
\begin{equation} \label{isometry-inv}
  \begin{split}
    \rho^{(N)}(A \mathbf{x}_1,\dots,A
    \mathbf{x}_N)&=\rho^{(N)}(\mathbf{x}_1,\dots,\mathbf{x}_N) \\ &\phantom{= }
    \text{for all}
    \ \mathbf{x}_i \in \Sph^d, \ A \in \SO(d+1) .
  \end{split}
\end{equation} 
By summation over permutations and integration over isometries, joint densities
satisfying \eqref{permutation-inv} and \eqref{isometry-inv} do exist.  In this
case we obtain
\begin{align}
\mathbb{E}\X_N(B)&=N\sigma(B), \label{eq:expect}\\
\mathbb{V}\X_N(B)
&=\mathbb{E}(\X_N(B)^2)-(\mathbb{E}\X_N(B))^2\notag\\
\begin{split} \label{eq:var}
&=N\sigma(B)(1-\sigma(B)) \\
&\phantom{=}+N(N-1)\iint_{B\times B}
\left(\rho_2^{(N)}(\mathbf{x}_1,\mathbf{x}_2)-1
\right)\,\dd\sigma(\mathbf{x}_1)\,\dd\sigma(\mathbf{x}_2).
\end{split}
\end{align}
The variance is independent of the position and orientation of the test set
$B$. So for a spherical cap the number variance only depends on the radius of
the cap.

\subsection*{Determinantal Point Processes}

Following~\cite{Hough_Krishnapur_Peres+2009:zeros_gaussian}, we introduce
determinantal point processes on $\Sph^d$. As pointed out before, we formulate
the description in terms of joint densities, rather than joint intensities.

\begin{definition}
  A simple point process on $\Sph^d$ is called determinantal with kernel~$K$ if
  its joint densities with respect to $\sigma$ are given by
\begin{align}\label{DPPdensities}
\rho^{(N)}_k(\mathbf{x}_1,\dots,\mathbf{x}_k) = \frac{(N-k)!}{N!}
\det \left(K(\mathbf{x}_i,\mathbf{x}_j)\right)_{i,j=1}^k,
\qquad 1 \leq k \leq N. 
\end{align} 
\end{definition}

From the definition, permutations of the variables do not change the
process. Furthermore, if $\mathbf{x}_i=\mathbf{x}_j$ for some $i\neq j$, then
the density is zero.

In \cite{Hough_Krishnapur_Peres+2009:zeros_gaussian} it is shown that
a process $\X_N$ samples exactly $N$ points if and only if it is
associated with the projection of $L^2$ to an $N$-dimensional
subspace $H$. Let $\psi_1,\ldots,\psi_N$ be an orthonormal basis of
$H$, then the kernel is given by
\begin{equation}
  \label{eq:projkernel}
  K_H(\mathbf{x},\mathbf{y})=\sum_{i=1}^N\psi_i(\mathbf{x})
  \overline{\psi_i(\mathbf{y})}.
\end{equation}


%% file: Definition.tex
\section{Hyperuniformity on the Sphere}\label{sec:sphere}
Complementing the extensive study of the notion of hyperuniformity in the
infinite setting, we are interested in studying an analogous property of
sequences of point sets in compact spaces. For convenience, we study the
$d$-dimensional unit sphere $\mathbb{S}^d$. Our ideas immediately generalise to
homogeneous spaces; further generalisations might be more elaborate,
since we rely heavily on harmonic analysis and specific properties of special
functions. For instance, for the flat torus a similar study has been carried
out in \cite{Stepanyuk2020:hyperuniform_point_sets}.

In order to adapt to the compact setting, we replace the infinite set $X$
studied in the classical notion of hyperuniformity by a sequence of finite
point sets, $(X_N)_{N\in \J}$, where we assume that the cardinality $\#X_N$ is
$N$. By using an infinite set $\J\subseteq\mathbb{N}$ as index set, we always
allow for subsequences.

Throughout the paper we use the notation
\begin{equation*}
  C(\mathbf{x},\phi)=\{\mathbf{y}\in\mathbb{S}^d\mid
  \langle\mathbf{x},\mathbf{y}\rangle>\cos\phi\}
\end{equation*}
for the (open) spherical cap with center $\mathbf{x}$ and opening angle
$\phi$. The normalised surface area of the cap is given by
\begin{equation} \label{eq:normalised.surface.cap}
  \sigma\left(C(\mathbf{x},\phi)\right)=
  \gamma_d\int_0^\phi\sin(\theta)^{d-1}\dd\theta\asymp\phi^{d} \quad\text{as
  }\phi\to0,
\end{equation}
where
\begin{equation*}
  \gamma_d=\left(\int_0^\pi\sin(\theta)^{d-1}\dd\theta\right)^{-1}=
  \frac{\Gamma(d)}{2^{d-1}\Gamma(d/2)^2}.
\end{equation*}
Notice that $\gamma_d=\frac{\omega_{d-1}}{\omega_d}$, where $\omega_d$ is the
surface area of $\mathbb{S}^d$.
Here and throughout the paper, we shall use $f(x) \asymp g(x)$ as $x \to x_0$
to mean that there exist positive constants $c$ and $C$ such that
$c \, g(x) \leq f(x) \leq C \, g(x)$ for $x$ sufficiently close to~$x_0$.  We
will write $\sigma(C(\phi))$ for the normalised surface area of the cap
$C(\cdot,\phi)$.

In this paper we shall study the \emph{number variance}.
  \begin{definition}[Number variance]\label{def:number-v}
    Let $\X_N$ be a point process on the sphere~$\mathbb{S}^d$ sampling
    $N$ points. The \emph{number variance} of $\X_N$ for caps of opening
    angle $\phi$ is given by
  \begin{equation}
    \label{eq:variance}
    V(\X_N,\phi):=\mathbb{V}\X_N(C(\cdot,\phi)):=
\mathbb{E}\left(\X_N(C(\cdot,\phi))^2\right)-
\left(\mathbb{E}\X_N(C(\cdot,\phi))\right)^2.
  \end{equation}
  If the process $\X_N$ is rotation invariant, the implicit
  integration with respect to the center of the cap $C(\cdot,\phi)$
  can be omitted.
  \end{definition}

As in the Euclidean case we define hyperuniformity by a comparison
between the behaviour of the number variance of a sequence of point
sets and of the i.i.d. case. For i.i.d. random points, the variance is
$N\sigma(C(\phi))(1-\sigma(C(\phi)))$ (see \eqref{eq:var}), which has order of magnitude $N$, $N\sigma(C(\phi_N))$, and $t^d$, respectively, in
the three cases \eqref{eq:large}, \eqref{eq:small}, and
\eqref{eq:threshold} listed below.

\begin{definition}[Hyperuniformity]\label{def-hyper}
  Let $\X_N$ be a  point process on the sphere~$\mathbb{S}^d$ sampling $N$ points. The  process $(\X_N)_{}$ is called
  \begin{itemize}
  \item \textbf{hyperuniform for large caps} if
    \begin{equation}
      \label{eq:large}
      V(\X_N,\phi)=o\left(N\right)\quad \text{as } N\to\infty
    \end{equation}
 for all $\phi\in(0,\frac\pi2)$ ;
\item \textbf{hyperuniform for small caps} if 
\begin{equation}
      \label{eq:small}
      V(\X_N,\phi_N)=o\left(N\sigma(C(\phi_N))\right)
      \quad \text{as } N\to\infty
    \end{equation}
      and all sequences $(\phi_N)_{N\in\mathbb{N}}$ such that
  \begin{enumerate}
  \item $\lim_{N\to\infty}\phi_N=0$
  \item $\lim_{N\to\infty}N\sigma(C(\phi_N))=\infty$,
      which is equivalent to ${\phi_NN^{\frac1d}\to\infty}$ as ${N\to\infty}$.
  \end{enumerate}
\item \textbf{hyperuniform for caps at threshold order} if
  \begin{equation}
    \label{eq:threshold}
    \limsup_{N\to\infty}V(\X_N,tN^{-\frac1d})=
  \mathcal{O}(t^{d-1}) \quad\text{as } t\to\infty.
\end{equation}
The $\mathcal{O}(t^{d-1})$ in \eqref{eq:threshold} could be replaced by the
less strict $o(t^d)$ in a more general setting.
\end{itemize}
\end{definition}


%% file: Intersection.tex
\section{Intersection Volume of Spherical Caps}
\label{sec:intersection}

In this section we collect some formulas and properties of the intersection
volume of two spherical caps that will be needed in the discussion
later on. Besides a possibly new formula for the volume of the intersection of
two caps of equal size we provide sharp inequalities and asymptotic expansions,
which enable us to obtain precise results on the number variance.

We will briefly introduce some basic facts and notation regarding spherical
harmonics.  Let $\mathcal H_\ell$ denote the vector space of spherical
harmonics of degree $\ell\in \mathbb{N}$. Its dimension is
\begin{equation*}
  Z(d,\ell) = \frac{2\ell+d-1}{d-1}\binom{\ell+d-2}{d-2}.
\end{equation*}
With respect to the $L^2(\Sph^d,\sigma)$ inner product, $\mathcal H_\ell$ has a real 
orthonormal basis $\{Y_{\ell,k}\}_{k=1}^{Z(d,\ell)}$. The addition theorem for
spherical harmonics (see~\cite{Mueller1966:spherical_harmonics}) gives
\begin{align*}
  \sum_{k=1}^{Z(d,\ell)} Y_{\ell,k}(\mathbf{x})Y_{\ell,k}(\mathbf{y}) = Z(d,\ell)
  P_\ell^{(d)}(\langle \mathbf{x},\mathbf{y} \rangle) , \quad
  \mathbf{x},\mathbf{y} \in \Sph^d ,
\end{align*}
where $P_\ell^{(d)}$, $\ell \geq 0$, are the Legendre polynomials for the sphere $\Sph^d$
normalised by $P_\ell^{(d)}(1)=1$. Notice that for $d\geq2$ these are
Gegenbauer polynomials for the parameter $\frac{d-1}2$:
\begin{equation} \label{eq:interconnection.formula}
  Z(d,\ell)P_\ell^{(d)}(x)=\frac{2\ell+d-1}{d-1} C_\ell^{\frac{d-1}2}(x).
\end{equation}

It is well-known that the Laplace series for the indicator function of the spherical cap
$C(\mathbf{x},\phi)$ is given by
\begin{equation*}
  \mathbbm{1}_{C(\mathbf{x},\phi)}(\mathbf{y})=
  \sigma(C(\cdot,\phi))+\sum_{n=1}^\infty a_n(\phi)Z(d,n)
  P_n^{(d)}(\langle\mathbf{x},\mathbf{y}\rangle),
\end{equation*}
where the Laplace coefficients are given by
\begin{equation}\label{eq:an-phi}
  a_n(\phi)=\gamma_d\int_0^\phi
  P_n^{(d)}(\cos(\theta))\sin(\theta)^{d-1}\,\dd\theta= 
  \frac{\gamma_d}d\sin(\phi)^dP_{n-1}^{(d+2)}(\cos(\phi))
\end{equation}
for $ n\geq1$.
The intersection volume is then obtained as the spherical convolution of the
indicator function with itself. This gives
\begin{equation}\label{eq:gphi}
  \begin{split}
    g_\phi(\langle\mathbf{x},\mathbf{y}\rangle)&:= \sigma(C(\mathbf{x},\phi)\cap
    C(\mathbf{y},\phi))- \sigma(C(\phi))^2\\
    &=\sum_{n=1}^\infty a_n(\phi)^2Z(d,n)
    P_n^{(d)}(\langle\mathbf{x},\mathbf{y}\rangle).
  \end{split}
\end{equation}

In \cite{Lee-Kim2014:concise_formulas_intersection} formulas for the volume of
the intersection of two spherical caps have been derived. In our special case
of the intersection of two caps of equal size, we get
\begin{equation*}
  \sigma(C(\mathbf{x},\phi)\cap C(\mathbf{y},\phi))=
\frac{d-1}\pi\int_{\frac\psi2}^\phi\sin(t)^{d-1}
\int_0^{\arccos(\frac{\tan(\frac\psi2)}{\tan(t)})}\sin(u)^{d-2}\,\dd u\,\dd t,
\end{equation*}
where $\langle\mathbf{x},\mathbf{y}\rangle=\cos\psi$ and
$\psi\leq2\phi$.

The change of variables
\begin{align*}
  \tan(v)&=\tan(t)\cos(u), \\
  \sin(w)&=\sin(t)\sin(u)
\end{align*}
transforms the double integral into
\begin{equation*}
\frac1\pi\int_{\frac\psi2}^\phi
\frac{(\sin^2\phi-\sin^2v)^{\frac{d-1}2}}{\cos(v)^{d-1}}\,\dd v.
\end{equation*}
This gives
\begin{equation}\label{eq:diff-cap-int}
  g_\phi(1)-g_\phi(\cos\psi)=
  \sigma(C(\mathbf{x},\phi)\setminus C(\mathbf{y},\phi))=
    \frac1\pi\int_0^{\frac\psi2}
\frac{\left(\sin^2\phi-\sin^2v\right)_+^{\frac{d-1}2}}{\cos(v)^{d-1}}\,\dd v
\end{equation}
for all $0<\psi<\pi$; here we define $(a)_+:=\max(0,a)$.

From this we obtain the following lemma.
\begin{lemma}\label{lem5}
  There exists a positive constant $A_d$ depending only on $d$ such that for
  all $(\phi,\psi)$ with $0\leq\psi\leq2\phi\leq\pi$ the inequalities
  \begin{equation}
  \label{eq:hypercap}
 \frac1{2\pi}
 \psi(\sin\phi)^{d-1}- A_d\psi^3\sin(\phi)^{d-3} \leq
 \sigma(C(\mathbf{x},\phi)\setminus C(\mathbf{y},\phi))\leq\frac1{2\pi}
  \psi(\sin\phi)^{d-1}
\end{equation}
hold. Here, $\cos\psi=\langle\mathbf{x},\mathbf{y}\rangle$. 
For $d\leq3$, these
inequalities hold for $(\phi,\psi)\in[0,\frac\pi2]\times[0,\pi]$.
\end{lemma}
\begin{proof}
  In the integral \eqref{eq:diff-cap-int}
  we make the substitution $\sin(v)=\sin(\phi)\sin(u)$ to obtain
  \begin{equation}\label{eq:capsize}
    g_\phi(1)-g_\phi(\cos\psi)=\frac{\sin(\phi)^d}\pi
    \int_0^{h(\phi,\psi)}\frac{\cos(u)^d}{(1-\sin(\phi)^2\sin(u)^2)^{\frac d2}}
    \dd u,
  \end{equation}
  where the upper bound in integration is given by
  \begin{equation*}
    h(\phi,\psi)=\arcsin\left(\min\left(1,
        \frac{\sin(\frac\psi2)}{\sin(\phi)}\right)\right).
  \end{equation*}
  Now, the integrand in \eqref{eq:capsize} is bounded from below by
  $\cos(u)^d\geq 1-\frac d2u^2$, from which we derive the estimate
  \begin{equation*}
    g_\phi(1)-g_\phi(\cos\psi)\geq \frac{\sin(\phi)^d}\pi
    \left(\arcsin\left(\frac{\sin(\frac\psi2)}{\sin(\phi)}\right)-
      \frac d6\left(\arcsin\left(\frac{\sin(\frac\psi2)}{\sin(\phi)}\right)
      \right)^3\right).
  \end{equation*}
  From this we derive the lower bound in \eqref{eq:hypercap} using the
  estimates
  \begin{align*}
    x&\leq\arcsin(x)\leq x+\left(\frac\pi2-1\right)x^3\\
    x-\frac{x^3}3&\leq\sin(x)\leq x
  \end{align*}
  for $x\geq0$.

  For the upper bound, we just observe that the integrand in
  \eqref{eq:diff-cap-int} is bounded by $\sin(\phi)^{d-1}$.
\qed
\end{proof}

For $d=2$, we get the explicit formula
\begin{align*}
  &\sigma\left(C(\mathbf{x},\phi)\setminus C(\mathbf{y},\phi)\right)\\
  &=
  \begin{cases}
\frac1\pi\left(\arcsin\left(\frac{\sin\frac\psi2}{\sin\phi}\right)-
\arcsin\left(\frac{\tan\frac\psi2}{\tan\phi}\right)\cos\phi\right)
&\text{for }\psi\leq2\phi
\\
\sin^2\frac\phi2&\text{for }\psi>2\phi,
  \end{cases}
\end{align*}
where $\cos\psi=\langle\mathbf{x},\mathbf{y}\rangle$.

%% file: Spherical.tex
\section{The Spherical Ensemble}\label{sec:spherical-ensemble}
The spherical ensemble of $N$ points is obtained by stereographically
projecting the eigenvalues of $A^{-1}B$ to the sphere $\Sph^2$, where $A$ and
$B$ are $N\times N$ matrices with i.i.d. random complex Gaussian entries
(see~\cite[Section~4.3.8]{Hough_Krishnapur_Peres+2009:zeros_gaussian} or
\cite[Chapter~5]{Krishnapur2006:zeros_random_analytic}).

These eigenvalues form a determinantal point process $\X_N^S$ on $\mathbb{C}$
with kernel
\begin{equation}\label{eq:widetildeK}
  \widetilde K_N(z,w):=(1+z\overline{w})^{N-1}
\end{equation}
with respect to the measure
\begin{equation*}
  \dd\mu_N(z):=\frac{N}{\pi(1+|z|^2)^{N+1}}\,\dd\lambda_2(z),
\end{equation*}
where $\lambda_2$ denotes the Lebesgue measure on $\mathbb{C}$.
The corresponding function space is the space of square integrable entire
functions 
\begin{equation*}
  \mathscr{P}_N:=L^2(\mathbb{C},\dd\mu_N)\cap H(\mathbb{C}),
\end{equation*}
which consists exactly of the polynomials of degree $\leq N-1$. The kernel
$\widetilde K_N$ is the reproducing kernel of this Hilbert space.

Applying the stereographic projection
\begin{equation}\label{eq:stereo}
  \mathbf{x}=(x_1,x_2,x_3)\in\Sph^2\mapsto\frac{x_1+ix_2}{1-x_3}
\end{equation}
transforms the measure by
\begin{equation}\label{eq:mustereo}
  \dd\mu_N(z)=\frac N{2^{N-1}}(1-x_3)^{N-1}\dd\sigma(\mathbf{x}).
\end{equation}
In order to obtain an isometry of spaces $L^2(\mathbb{C},\mu_N)$ and
$L^2(\Sph^2,\sigma)$, the basis elements of $\mathscr{P}_N$ are mapped by
\begin{equation}\label{eq:zkmapsto}
  z^k\mapsto \frac{\sqrt N}{2^{\frac{N-1}2}}(x_1+ix_2)^k(1-x_3)^{\frac{N-1}2-k}
  \quad\text{for }k=0,\ldots,N-1.
\end{equation}

Inserting \eqref{eq:stereo} into \eqref{eq:widetildeK} and multiplying by
$\frac N{2^{N-1}}((1-x_3)(1-y_3))^{\frac{N-1}2}$ gives the projected kernel on
the sphere
\begin{equation*}
  K_N(\mathbf{x},\mathbf{y})=\frac{N}{2^{N-1}}
  \left(\frac{1+\langle\mathbf{x},\mathbf{y}\rangle-x_3-y_3+i(x_2y_1-x_1y_2)}
  {\sqrt{(1-x_3)(1-y_3)}}\right)^{N-1}
\end{equation*}
and the space of functions on $\Sph^2$ is spanned by the functions given in
\eqref{eq:zkmapsto}.  These functions are orthogonal with respect to $\sigma$.

In order to compute the expectation of a general energy sum with
respect to the process generated by $K_N$, we compute the determinant
\begin{equation*}
 N(N-1) \rho_2^{(N)}(\mathbf{x},\mathbf{y})=
K_N(\mathbf{x},\mathbf{x})K_N(\mathbf{y},\mathbf{y})-
\left|K_N(\mathbf{x},\mathbf{y})\right|^2.
\end{equation*}
We notice that $K_N(\mathbf{x},\mathbf{x})=N$ and
\begin{equation*}
  |1+z\overline{w}|^2=
  \left|\left(1+\frac{x_1+ix_2}{1-x_3}\frac{y_1-iy_2}{1-y_3}\right)\right|^2=
  \frac1{(1-x_3)(1-y_3)}\left(2+2\langle\mathbf{x},\mathbf{y}\rangle\right)
\end{equation*}
for $\mathbf{x},\mathbf{y}\in\Sph^2$, which gives
\begin{equation*}
  \left|K_N(\mathbf{x},\mathbf{y})\right|^2=
  N^2\left(\frac{1+\langle\mathbf{x},\mathbf{y}\rangle}2\right)^{N-1}
\end{equation*}
and
\begin{equation*}
N(N-1) \rho_2^{(N)}(\mathbf{x},\mathbf{y})=
N^2\left(1-\left(\frac{1+\langle\mathbf{x},\mathbf{y}\rangle}2\right)^{N-1}
\right).
\end{equation*}
Now let $g:[-1,1]\to\mathbb{R}$ be a function with
$\int_{-1}^1g(x)\,\dd x=0$. Then
\begin{multline}
  E_{g}(N):=\mathbb{E}\sum_{i,j=1}^Ng\left(\langle\mathbf{x}_i,\mathbf{x}_j\rangle\right)\\
=Ng(1)+N^2\iint_{\mathbb{S}^2\times\mathbb{S}^2}
g\left(\langle\mathbf{x},\mathbf{y}\rangle\right)
\left(1-\left(\frac{1+\langle\mathbf{x},\mathbf{y}\rangle}2\right)^{N-1}\right)
\,\dd\sigma(\mathbf{x})\,\dd\sigma(\mathbf{y})\label{eq:g-energy}\\
=
\frac{N^2}2\int_{-1}^1\left(g(1)-g(x)\right)\left(\frac{1+x}2\right)^{N-1}\,\dd
x.
\end{multline}
We apply \eqref{eq:g-energy} to the function 
$g_\phi$ given by \eqref{eq:gphi}.
Putting everything together, we obtain
\begin{align}
V(\X_N^S,\phi)
&=E_{g_\phi}(N) \notag \\
\begin{split}
&=\frac{N^2}{4\pi}\sin\phi
\int_{-1}^1\arccos(x)\left(\frac{1+x}2\right)^{N-1} \dd x \\
&\phantom{=}+ \mathcal{O}\Bigg(\frac{N^2}{\sin\phi}\int_{-1}^1\arccos(x)^3
\left(\frac{1+x}2\right)^{N-1} \dd x\Bigg) 
\end{split} \\
&=\frac{\sin\phi}{2\sqrt\pi} \, \frac{\Gamma(N+\frac12)}{\Gamma(N)}+
\mathcal{O}\Big( \frac{1}{N^{1/2} \sin\phi} \Big) \notag \\
&=\frac{\sqrt{\sigma(C(\phi))(1-\sigma(C(\phi)))}}{\sqrt\pi}N^{1/2}+
\mathcal{O}\Big( \frac{1}{N^{1/2} \sin\phi} \Big)\label{eq:Eg-asymp}
\end{align}
valid for $\phi\in(0,\frac\pi2)$.  Thus, we have proved the following lemma. We
note that \eqref{eq:Eg-asymp} was obtained in \cite[Lemma~2.1]{AlishahiZamani}
with the restriction that $\sigma(C(\phi))^{-1}=o(N)$ and with a weaker error
term.
\begin{lemma}\label{lem1}
  The number variance of the spherical ensemble satisfies for ${\phi\in(0,\pi)}$
  \begin{equation}\label{eq:sphere-variance}
    V(\X_N^S,\phi)=\frac{\sqrt{\sigma(C(\phi))(1-\sigma(C(\phi)))}}
    {\sqrt\pi}N^{1/2}+\mathcal{O}\Big( \frac{1}{N^{1/2} \sin\phi} \Big)
  \end{equation}
  with an absolute implied constant; especially,
  \begin{equation}\label{eq:spher-threshold}
    \lim_{N\to\infty}V(\X_N^S,tN^{-\frac12})=
    \frac t{2\sqrt\pi}+\mathcal{O}(t^{-1}).
  \end{equation}
\end{lemma}
\begin{remark}
  Inserting \eqref{eq:diff-cap-int} directly into \eqref{eq:g-energy} gives the
  closed formula
  \begin{equation*}
    E_{g_\phi}( N ) = \frac{N \sin^2 \phi}{\pi} \int_0^1 \left( 1 - v^2 \right)^{\frac{1}{2}} \left( 1 - v^2 \, \sin^2 \phi \right)^{N-1} \dd v,
  \end{equation*}
  which could be used for an alternative yet slightly more elaborate proof of
  Lemma~\ref{lem1}.
\end{remark}
From this lemma we immediately obtain the following theorem.
\begin{theorem}\label{thm:spher-hyper}
  The spherical ensemble is hyperuniform in all three regimes.
\end{theorem}
\begin{proof}
  For the large cap case, we obtain $V(\X_N^S,\phi)=\mathcal{O}(N^{1/2})$;
  for the small cap case, we obtain
  $V(\X_N^S,\phi_N)=\mathcal{O}((N\phi_N)^{1/2})=o(N\phi_N)$. In the threshold
  order case, we use \eqref{eq:spher-threshold}.
  \qed
\end{proof}
\begin{remark}
  The error term in \eqref{eq:sphere-variance} has the correct order with
  respect to $N$ and $\phi$. This shows that taking $\phi_N=o(N^{-\frac12})$
  does not make sense, because then the error term would become the dominant term that tends
  to $\infty$.
\end{remark}

%% file: Harmonic.tex
\section{The Harmonic Ensemble}\label{sec:harmonic}
The function space of spherical harmonics of degree $\leq L$ on the sphere
$\Sph^d$ and the projection kernel to this space of dimension
$Z(d+1,L)=\frac{2L+d}{d}\binom{L+d-1}{d-1}$ was used in
\cite{Beltran_Marzo_Ortega-Cerda2016:determinantal} to define a determinantal
point process $\X_L^H$, the \emph{harmonic ensemble}. This process samples
$N:=N_L:=Z(d+1,L)\asymp L^d$ points. We will study this process with respect to
hyperuniformity in this section.

The projection kernel to this space is given by
\begin{align*}
  K_L(\langle\mathbf{x},\mathbf{y}\rangle):=
  \sum_{\ell=0}^L Z(d,\ell) P_\ell^{(d)}(\langle \mathbf{x},\mathbf{y} \rangle) =
  \frac{Z(d+1,L)}{\binom{L+d/2}{L}}
  \mathcal{P}_L^{(\frac{d}{2},\frac{d}{2}-1)}(\langle\mathbf{x},\mathbf{y}\rangle)
\end{align*}
for $\mathbf{x},\mathbf{y} \in \Sph^d$, where $\mathcal{P}_L^{(\alpha,\beta)}$,
$L \geq 0$, are the usual Jacobi polynomials. This identity follows from
\cite[Theorem~7.1.3]{Andrews-Askey-Roy1999:Special_functions} after rewriting
the Legendre polynomials in terms of Gegenbauer and Jacobi polynomials using
\eqref{eq:interconnection.formula}.

\begin{theorem}\label{thm:harmonic}
  The harmonic ensemble is hyperuniform for large and small caps. In the
  threshold order regime the weaker property
  \begin{equation}
    \label{eq:threshold-log}
    \limsup_{L\to\infty}V(\X_L^H,tN_L^{-\frac1d})=\mathcal{O}(t^{d-1}\log t)=o(t^d)
  \end{equation}
  holds.
\end{theorem}
\begin{proof}
  The number variance $V(\X_L^H,\phi)$ can be expressed as (cf. similar
  computations that lead to \eqref{eq:g-energy})
  \begin{equation*}
    \gamma_d \int_0^\pi (g_\phi(1)-g_\phi(\cos \theta))K_L(\cos \theta)^2
    (\sin \theta)^{d-1}\,\dd\theta,
  \end{equation*}
  where $g_\phi$ is given by
  \eqref{eq:gphi}. Using Lemma~\ref{lem5}, we obtain
  \begin{equation}\label{eq:harmonic_number_variance}
  \begin{split}
    &V(\X_L^H,\phi)\\
=& \frac{\gamma_d}{\pi} \left(\!\!\frac{Z(d+1,L)}{\binom{L+\frac d2}L}\!\right)^2
    \!\!\!\!(2\sin\phi)^{d-1}\!\!\!\int_0^{2\phi}\!\!\!\left(\mathcal{P}_L^{(\frac d2,\frac d2-1)}(\cos
      \theta)\right)^2\!\!\left(\sin\frac \theta2\right)^d\!\!
    \left(\cos\frac \theta2\right)^{d-1}\!\!\!\!\!\!\dd\theta\\+
    &\mathcal{O}\Bigg(L^d(\sin\phi)^{d-3}\int_0^{2\phi}\left(\mathcal{P}_L^{(\frac d2,\frac d2-1)}(\cos
        \theta)\right)^2\left(\sin\frac \theta2\right)^{d+2}
      \left(\cos\frac \theta2\right)^{d-1}\dd\theta\Bigg) \\+
    &\gamma_d \left(\frac{Z(d+1,L)}{\binom{L+\frac d2}L}\right)^2
    \sigma(C(\phi))\int_{2\phi}^\pi\left(\mathcal{P}_L^{(\frac d2,\frac d2-1)}(\cos
        \theta)\right)^2(\sin \theta)^{d-1}\,\dd\theta.
  \end{split}
\end{equation}

  The case of large and small caps was studied in
  \cite{Beltran_Marzo_Ortega-Cerda2016:determinantal}; we summarise the
  computations given there for completeness. The case of caps at threshold
  order is new and will be given in more detail.
  We make use of known asymptotic expansions for the Jacobi polynomials
  (see~\cite[5.2.3 and 5.2.4]{Magnus_Oberhettinger_Soni1966:formulas_theorems})
  \begin{align}
    \label{eq:P-asymp1}
    \mathcal{P}_L^{(\frac d2,\frac d2-1)}(\cos \theta)&=
    \frac{\cos\left(\left(L+\frac d2\right)\theta-\frac\pi4(d+1)\right)}
    {\sqrt{\pi L}\left(\sin\frac \theta2\right)^{\frac{d+1}2}
    \left(\cos\frac \theta2\right)^{\frac{d-1}2}}+\mathcal{O}(L^{-\frac32})\\
 \mathcal{P}_L^{(\frac d2,\frac d2-1)}\left(\cos \frac \tau L\right)&=L^{\frac d2}
 \left(\frac 2\tau\right)^{\frac d2}J_{\frac d2}(\tau)+\mathcal{O}(L^{\frac d2-1}),
    \label{eq:P-asymp2}               
  \end{align}
  where $J_{\frac d2}$ denotes the Bessel function of the first kind of index $\frac d2$.  Given
  a constant $C>0$, the asymptotic relation \eqref{eq:P-asymp1} is used for
  $\theta>\frac CL$, whereas relation \eqref{eq:P-asymp2} is used for
  $\theta=\frac\tau L\leq\frac CL$.

  This gives
  \begin{multline}
    \label{eq:int-1}
      \int_0^{\frac CL}\left(\mathcal{P}_L^{(\frac d2,\frac d2-1)}(\cos
        \theta)\right)^2\left(\sin\frac \theta2\right)^d \left(\cos\frac
        \theta2\right)^{d-1}\,\dd\theta\\= \frac 1L\int_0^C J_{\frac
        d2}(\theta)^2\,\dd\theta+\mathcal{O}(L^{-2})
  \end{multline}
  for the integral over the ``small'' values of $\theta$,
  \begin{multline}
    \label{eq:int-2}
    \int_{\frac CL}^{\alpha}\left(\mathcal{P}_L^{(\frac d2,\frac d2-1)}(\cos
      \theta)\right)^2\left(\sin\frac \theta2\right)^d
    \left(\cos\frac \theta2\right)^{d-1}\,\dd\theta\\=
    \frac1{\pi L}\int_{\frac CL}^{\alpha}\frac{\cos\left(\left(L+\frac
          d2\right)\theta-\frac\pi4(d+1)\right)^2}{\sin(\frac
      \theta2)}\,\dd\theta+\mathcal{O}(L^{-2})
  \end{multline}
  for the integral over the ``large'' values of $\theta$,
  \begin{multline}
    \label{eq:int-3}
    \int_{\alpha}^\pi\left(\mathcal{P}_L^{(\frac d2,\frac d2-1)}(\cos
      \theta)\right)^2\left(\sin\frac \theta2\right)^{d-1} \left(\cos\frac
        \theta2\right)^{d-1}\,\dd\theta\\
    =\frac1{\pi L}\int_{\alpha}^\pi
    \frac{\cos\left(\left(L+\frac
          d2\right)\theta-\frac\pi4(d+1)\right)^2}{(\sin\frac \theta2)^2}\,\dd\theta+
    \mathcal{O}(L^{-2})=\mathcal{O}((L\alpha)^{-1})
  \end{multline}
  and
  \begin{multline}
    \label{eq:int-4}
    \int_0^\alpha \left(\mathcal{P}_L^{(\frac d2,\frac d2-1)}(\cos
      \theta)\right)^2\left(\sin\frac \theta2\right)^{d+2}
    \left(\cos\frac \theta2\right)^{d-1}\,\dd t\\=
    \frac1{\pi L}\int_0^\alpha \!\!\cos\left(\left(L+\frac
        d2\right)\theta-\frac\pi4(d+1)\right)^2 \sin\Big(\frac
    \theta2\Big)\,\dd\theta+\mathcal{O}(L^{-2})
    =
    \mathcal{O}(\alpha^2L^{-1})
  \end{multline}
for the integral in the error term.
  
In the case of large caps ($0<\phi<\frac\pi2$ fixed), we compute the number
variance as
  \begin{multline*}
    V(\X_L^H,\phi)=\frac{\gamma_d}{\pi}\left(\frac{Z(d+1,L)}{\binom{L+\frac d2}L}\right)^2
    \frac{(2\sin\phi)^{d-1}}L\Biggl(\int_0^C J_{\frac d2}(\theta)^2\,
      \dd\theta\\+
    \frac1{\pi}\int_{\frac CL}^{2\phi}\frac{\cos\left(\left(L+\frac
          d2\right)\theta-\frac\pi4(d+1)\right)^2}{\sin(\frac
      \theta2)}\,\dd\theta+
    \mathcal{O}(L^{-1})+\mathcal{O}(\phi^{-1})\Biggr)\\=
  \mathcal{O}((\sin\phi)^{d-1}L^{d-1}\log( L \phi ) ),
\end{multline*}
where we have used $\left(Z(d+1,L)/\binom{L+\frac d2}L\right)^2 \asymp L^d$ and
the logarithmic term comes from the second summand. This is the true asymptotic
order and due to $N_L\asymp L^d$ we have $V(\X_L^H,\phi)=o\left(N_L\right) $ as
$L\to\infty$ for all $\phi\in(0,\frac\pi2)$.

In the case of small caps, a similar computation gives
\begin{align*}
  V(\X_L^H,\phi_L) &= \mathcal{O}((\sin\phi_L)^{d-1}L^{d-1}\log( L \phi_L ) ) = \mathcal{O}\Big( \frac{\log( L \phi_L )}{ L \phi_L } \, (\sin\phi_L)^{d}L^{d}\Big) \\ &= \mathcal{O}\Big( \frac{\log( L \phi_L )}{ L \phi_L } \, N_L\,\sigma(C(\phi_L)) \Big)=o(N_L\,\sigma(C(\phi_L))).
\end{align*}

For caps at threshold order, we analyse the three terms in
\eqref{eq:harmonic_number_variance} separately. The first term gives
\begin{equation*}
  \frac{\gamma_d}{\pi}\left(\frac{Z(d+1,L)}{\binom{L+\frac d2}L}\right)^2
    \frac{(2\sin tL^{-1})^{d-1}}L\left(\int_0^t J_{\frac d2}(\theta)^2\,\dd\theta+
    \mathcal{O}(L^{-1})\right)
\end{equation*}
using \eqref{eq:P-asymp2}.
We use the asymptotic behaviour of the Bessel function for
$\theta\to\infty$
(see~\cite[3.14.1]{Magnus_Oberhettinger_Soni1966:formulas_theorems})
\begin{equation*}
  J_{\frac d2}(\theta)=\frac{\cos\left(\theta-\frac{\pi (d+1)}4\right)}
  {\sqrt{\frac{\pi\theta}2}}+\mathcal{O}(\theta^{-\frac32})
\end{equation*}
to obtain
\begin{equation*}
  \int_0^tJ_{\frac d2}(\theta)^2\,\dd\theta=\frac1\pi\log t+\mathcal{O}(1).
\end{equation*}
Thus the first term in \eqref{eq:harmonic_number_variance} for $\phi_L\sim
tL^{-1}$ is of order $t^{d-1}\log t$.

The second term in \eqref{eq:harmonic_number_variance} is analysed using
\eqref{eq:int-4}, which gives an order of $t^{d-1}$ for $\phi_L\sim
tL^{-1}$. Similarly, the third term is analysed using \eqref{eq:int-3}, which
again gives an order of $t^{d-1}$.

Putting these order estimates together yields
\begin{equation*}
  V(\X_L^H,tL^{-1})=\mathcal{O}(t^{d-1}\log t)
\end{equation*}
and concludes our proof.
\qed
\end{proof}


%% file: Jittered.tex
\section{Jittered Sampling}\label{sec:jittered}

In \cite{Gigante_Leopardi2017:diameter_ahlfors} it is shown that on arbitrary
Ahlfors regular metric measure spaces there exist area-regular partitions. For
the case of the sphere studied here, an area regular partition is given by
$\mathcal{A}$ $=\{A_1,\ldots,A_N\}$ with $\bigcup_{i=1}^N A_i =\Sph^d$,
$\sigma(A_i)=\frac1N$, and $i\neq j\Rightarrow A_i\cap A_j=\emptyset$
satisfying
\begin{equation}\label{eq:diam}
  \diam(A_i) \le C_dN^{-1/d},\qquad i=1,\ldots,N,
\end{equation}
with a constant depending only on $d$
(see also \cite{alexander1972sum,bourgain1988distribution,
  kuijlaars1998asymptotics,mhaskar2001spherical,Le2006}).

Such partitions allow us to consider the average behaviour of \emph{jittered
  sampling}; the point process $\X_N^{\mathcal{A}}$ constructed by sampling the
sphere with the condition that each of the $N$ points lies in a distinct region
of the partition.

The jittered sampling variance integral is written as:
\begin{multline*}
  V(\X_N^{\mathcal{A}},\phi)\\=\!\!
  \int\limits_{\mathbb{S}^d}\!\int\limits_{A_1}\!\!\dots\!\! \int\limits_{A_N}\!\!
  \left(\sum_{i=1}^N
\mathbbm{1}_{C(\mathbf{x},\phi)}(\mathbf{y}_i)-N \sigma(C(\phi))\right)^2\!\!\!
\dd\sigma_1(\mathbf{y}_1)\cdots \dd\sigma_N(\mathbf{y}_N)\,\dd\sigma(\mathbf{x}),
\end{multline*}
where $\sigma_i(\cdot):=N\sigma(\cdot\cap A_i)$ is the uniform probability
measure on $A_i$.  
%
%
The integral can be split into off-diagonal and diagonal terms
\begin{align}
  &V(\X_N^{\mathcal{A}},\phi)
=\sum_{\substack{i,j=1\\i\neq j}}^N
  \int\limits_{A_i}\int\limits_{A_j}\!\!\sigma(C(\mathbf{y}_i,\phi)\cap C(\mathbf{y}_j,\phi))
  \,\dd\sigma_i(\mathbf{y}_i)\,\dd\sigma_j(\mathbf{y}_j)
  \label{eq:variance_jittered}\\&+
  N\sigma(C(\phi))-N^2\sigma(C(\phi))^2\notag\\
  =&N\sum_{i=1}^N\int\limits_{A_i} \Bigg(\int\limits_{\mathbb{S}^d}\!\!
    \sigma(C(\mathbf{y}_i,\phi)\cap
    C(\mathbf{y},\phi))\,\dd\sigma(\mathbf{y}) \notag\\ &-
    \int\limits_{A_i}\!\!\sigma(C(\mathbf{y}_i,\phi)\cap C(\mathbf{y},\phi))
    \,\dd\sigma(\mathbf{y})\Bigg) \dd\sigma_i(\mathbf{y}_i)+
  N\sigma(C(\phi))-N^2\sigma(C(\phi))^2\notag\\
    =&N^2\left(\int\limits_{\mathbb{S}^d}\int\limits_{\mathbb{S}^d}
    \sigma(C(\mathbf{x},\phi)\cap C(\mathbf{y},\phi))
\,\dd\sigma(\mathbf{x})\,\dd\sigma(\mathbf{y})-\sigma(C(\phi))^2\right)\notag\\
  &+\sum_{i=1}^N\int\limits_{A_i}\int\limits_{A_i}\Big(\sigma(C(\mathbf{x}_i,\phi))-
      \sigma(C(\mathbf{x}_i,\phi)\cap C(\mathbf{y}_i,\phi))\Big)
      \,\dd\sigma_i(\mathbf{x}_i)\,\dd\sigma_i(\mathbf{y}_i)\notag\\
  =&\frac12\sum_{i=1}^N\int\limits_{A_i}\int\limits_{A_i}
      \sigma(C(\mathbf{x}_i,\phi)\triangle C(\mathbf{y}_i,\phi))
      \,\dd\sigma_i(\mathbf{x}_i)\,\dd\sigma_i(\mathbf{y}_i),\label{eq:diagonal}
\end{align}
where $\triangle$ denotes the symmetric difference operator of two sets. For
the last equality, we have used
\begin{align*}
  &\int\limits_{\mathbb{S}^d}\int\limits_{\mathbb{S}^d}
    \sigma(C(\mathbf{x},\phi)\cap C(\mathbf{y},\phi))
    \,\dd\sigma(\mathbf{x})\,\dd\sigma(\mathbf{y})\\
&=\int\limits_{\mathbb{S}^d}\int\limits_{\mathbb{S}^d}\int\limits_{\mathbb{S}^d}
\mathbbm{1}_{C(\mathbf{x},\phi)}(\mathbf{z})
\mathbbm{1}_{C(\mathbf{y},\phi)}(\mathbf{z})\,
\dd\sigma(\mathbf{z})\,\dd\sigma(\mathbf{x})\,\dd\sigma(\mathbf{y})\\
 &= \int\limits_{\mathbb{S}^d}\int\limits_{\mathbb{S}^d}\int\limits_{\mathbb{S}^d}
\mathbbm{1}_{C(\mathbf{z},\phi)}(\mathbf{x})
\mathbbm{1}_{C(\mathbf{z},\phi)}(\mathbf{y})\,
\dd\sigma(\mathbf{x})\,\dd\sigma(\mathbf{y})\,\dd\sigma(\mathbf{z})
=\sigma(C(\phi))^2.
\end{align*}
So in fact the variance of the jittered sampling process reduces to
the diagonal terms. The measure of the symmetric difference can be
bounded like
\begin{equation*}
   \sigma(C(\mathbf{x}_i,\phi)\triangle C(\mathbf{y}_i,\phi))\leq
   C_d\operatorname{surface}(\partial C(\mathbf{x}_i,\phi))
   \arccos(\langle\mathbf{x}_i,\mathbf{y}_i\rangle),
 \end{equation*}
 where $C_d$ is a constant only depending on the dimension $d$. This can be
 seen from Lemma~\ref{lem5} using the observation that
 $\operatorname{surface}(\partial C(\mathbf{x}_i,\phi))\asymp
 (\sin\phi)^{d-1}$.
 
 From the diameter bounds coming from our choice of equipartition, every
 summand in \eqref{eq:diagonal} can be bounded by
 $\mathcal{O}(\phi^{d-1}N^{-\frac1d})$, which gives
\begin{equation}
  \label{eq:jittered-bound}
  V(\X_N^{\mathcal{A}},\phi)=\mathcal{O}\left(\phi^{d-1}N^{\frac{d-1}d}\right);
\end{equation}
the implied constant depends only on the dimension and the constants in
\eqref{eq:diam}. 

\begin{theorem}\label{thm:jittered}
  The jittered sampling point process is hyperuniform in all three regimes.
\end{theorem}
\begin{proof}
  From \eqref{eq:jittered-bound} it is now immediate that
  $V(\X_N^{\mathcal{A}},\phi)=o(N)$ for all $\phi\in(0,\frac\pi2)$, which proves
  hyperuniformity for large caps. 

Again from \eqref{eq:jittered-bound} we obtain
\begin{equation*}
  V(X_N,\phi_N)=\mathcal{O}\Big((\phi_NN^{\frac1d})^{d-1}\Big)=
  o\Big((\phi_NN^{\frac1d})^d\Big)=o\big(\phi_N^dN\big)
\end{equation*}
under the assumptions on $(\phi_N)_{N\in\mathbb{N}}$ in
Definition~\ref{def-hyper}, which proves hyperuniformity for small caps.

\label{sec:jitt-threshold}
Inserting $\phi_N=tN^{-\frac1d}$ into \eqref{eq:jittered-bound} yields
\begin{equation*}
  V(\X_N^{\mathcal{A}},tN^{-\frac1d})=\mathcal{O}\big(t^{d-1}\big)\qquad\text{
 as }t\to\infty,
\end{equation*}
which implies hyperuniformity at threshold order.
\qed
\end{proof}

\subsection*{Jittered Sampling is Determinantal}
In this subsection we consider a general probability space
$(\Lambda,\mathcal{B},\mu)$ with the only requirement that for any
$N\in\mathbb{N}$ there exists a partition of $\Lambda$ into $N$ disjoint
measurable sets of equal measure.  Consider an area-regular partition
$\mathcal{A}=\{A_1,\ldots,A_N\}$ of the space $\Lambda$ into pairwise disjoint
measurable sets; i.e., 
\begin{align*} 
A_i\cap A_j &= \emptyset, \qquad i\neq j,\\
\mu\Big(\bigcup_{i=1}^NA_i\Big)&=1, \\
\mu(A_i)&=\frac1N, \qquad i = 1, \dots, N.
\end{align*} Define the projection operator
\begin{equation*} p_{\mathcal{A}}(f)(x):=
\sum_{i=1}^N\frac{\Ind_{A_i}(x)}{\mu(A_i)}\int_{A_i}f(y)\,\dd\mu(y)=
\int_\Lambda K_{\mathcal{A}}(x,y)f(y)\,\dd\mu(y)
\end{equation*} to the space of functions measurable with respect to
the finite $\sigma$-algebra generated by $\mathcal{A}$. The kernel of this operator
is given by
\begin{equation*}
K_{\mathcal{A}}(x,y):=\sum_{i=1}^N\frac{\Ind_{A_i}(x)\Ind_{A_i}(y)}{\mu(A_i)}.
\end{equation*} The determinantal point process $\X_N^{\mathcal{A}}$
defined by the projection kernel $K_{\mathcal{A}}$ is then equal to
the jittered sampling process associated to the partition
$\mathcal{A}$, which can be seen by computing
\begin{multline*} \mathbb{E}\X_N^{\mathcal{A}}(A_1)\cdots
\X_N^{\mathcal{A}}(A_N)\\= \int_{A_1}\cdots\int_{A_N}
\det\left(K_{\mathcal{A}}(x_i,x_j)_{i,j=1}^N\right)\,\dd\mu(x_1)
\cdots \dd\mu(x_N).
\end{multline*} Expanding the determinant gives
\begin{multline*} \mathbb{E}\X_N^{\mathcal{A}}(A_1)\cdots
\X_N^{\mathcal{A}}(A_N)\\= \sum_\pi\sgn(\pi)\int_{A_1}\cdots\int_{A_N}
\prod_{i=1}^NK_{\mathcal{A}}(x_i,x_{\pi(i)})\,\dd\mu(x_1)\cdots\dd\mu(x_N).
\end{multline*} Now we notice that $K_{\mathcal{A}}(x_i,x_j)=0$ if
$i\neq j$ and $x_i\in A_i$ and $x_j\in A_j$. Thus, the integrand in the
sum vanishes for all $\pi\neq\id$, which gives
\begin{equation}\label{eq:jexpect} 
\mathbb{E}\X_N^{\mathcal{A}}(A_1)\cdots\X_N^{\mathcal{A}}(A_N)= \prod_{i=1}^N
\int_{A_i}K_{\mathcal{A}}(x_i,x_i)\,\dd\mu(x_i)=1.
\end{equation} The process $\X_N^{\mathcal{A}}$ samples $N$ points almost
surely by the fact that $NK_{\mathcal{A}}$ defines a projection operator
(see~\cite{Hough_Krishnapur_Peres+2009:zeros_gaussian}).  The product of random
variables $\X_N^{\mathcal{A}}(A_1)\cdots \X_N^{\mathcal{A}}(A_N)$ is either $0$
or $1$  and thus equal to $1$ (a.s.) by \eqref{eq:jexpect}. This implies that
the process samples exactly one point per set of the partition
$\mathcal{A}$. Furthermore, we have
\begin{align*}
\mathbb{E}\X_N^{\mathcal{A}}(D)&=\int_{D}K(x,x)\,\dd\mu(x)=
\sum_{i=1}^N\int_D\frac{\Ind_{A_i}(x)^2}{\mu(A_i)}\,\dd\mu(x)\\&=
\sum_{i=1}^N \frac{\mu(A_i\cap D)}{\mu(A_i)}=N\mu(D),
\end{align*} 
and, for $D\subseteq A_i$, this implies
$\mathbb{E}\X_N^{\mathcal{A}}(D)=\mu(D)/\mu(A_i)$; the sample point
chosen from $A_i$ is distributed with measure $\mu_i$ on $A_i$, where
\begin{equation*}
  \mu_i(A)=N\mu(A_i\cap A).
\end{equation*}
